\documentclass[12pt,a4paper]{amsart}

\usepackage{amscd}
\usepackage{graphicx}
\usepackage{amssymb,amsfonts,latexsym}

\input xy
\xyoption {all}

\newtheorem{thm}{Theorem}[section]
\newtheorem{cor}[thm]{Corollary}
\newtheorem{lem}[thm]{Lemma}
\newtheorem{prop}[thm]{Proposition}
\theoremstyle{definition}

\newtheorem{defn}[thm]{Definition}
\theoremstyle{remark}

\numberwithin{equation}{section}

\long\def\forget#1\forgotten{}

\newcommand{\mQ}{\mathbb{Q}}

\newcommand{\mS}{\mathcal{S}}

\newcommand{\mG}{\mathcal{G}}

\newcommand{\F}{\mathbb{F}}

\newcommand{\pst}{\widetilde{\psi}}

\newcommand\wh[1]{{\widehat{#1}}}
\newcommand{\ra}{\rightarrow}

\def\({\left(}
\def\){\right)}

\newcommand\oline[1] {{\overline{#1}}}

\newcommand\Hom{{\operatorname{Hom}}}

\renewcommand\Im{{\operatorname{Im}}}

\begin{document}

\title%[] %% short title
{On semiabelian $p$-groups}%

\def\Tech{Department of Mathematics, Technion-Israel Institute of Technology, Haifa 32000, Israel}

\author{ Danny Neftin}
\address{\Tech}
\email{neftind@tx.technion.ac.il}%

\begin{abstract}
The family of semiabelian $p$-groups is the minimal family that contains $\{1\}$ and is closed under quotients and semidirect products with finite abelian $p$-groups. Kisilevsky and Sonn have solved the  minimal ramification problem for a certain subfamily $\mG_p$ of the family of semiabelian $p$-groups. We show that $\mG_p$ is in fact the entire family of semiabelian $p$-groups and by this complete their solution to all semiabelian $p$-groups.
\end{abstract}

\maketitle

\section{introduction}

This paper is motivated by the minimal ramification problem for $p$-groups. Given a $p$-group $G$ it is an open problem to find the minimal number of primes ramified in a $G$-extension of $\mQ$ (see \cite{P}). As a consequence of Minkowski's Theorem this number is greater or equal to $d(G)$, the minimal number of generators of $G$. In \cite{KS}, Kisilevsky and Sonn  proved this number is exactly $d(G)$ for a family of $p$-groups denoted by  $\mG_p$ and defined as follows: %In order to describe $\mG_p$ let us first recall the notion of regular (or standard) wreath products.
\begin{defn} Let $\mG_p$ be the minimal family that satisfies:
\begin{enumerate}
\item any abelian $p$-group is in $\mG_p$,

\item if $H,G\in \mG_p$ then the standard wreath product $H\wr G$ is also in $\mG_p$,

\item if $G\in \mG_p$ and $G\ra \Gamma$ is a rank preserving epimorphism, i.e. with $d(G)=d(\Gamma)$,  then $\Gamma\in \mG_p$.
\end{enumerate}\end{defn}

This family is contained in the family of semiabelian groups (see \cite{KS}):

\begin{defn}
The family of {\it semiabelian} groups $\mS$ is the minimal family of groups that satisfies:
\begin{enumerate}
\item $\{1\}\in \mS$,

\item if $A$ is a finite abelian group and $H\in \mS$ acts on $A$ then the semidirect product $A\rtimes H$ is in $\mS$,

\item if $G\in \mS$ and $G\ra \Gamma$ is an epimorphism then $\Gamma\in \mS$.
\end{enumerate}
\end{defn}

We shall prove that $\mG_p$ is precisely the family of semiabelian $p$-groups. In fact this is an immediate corollary of the following theorem: \begin{thm}\label{main.thm} Let $G$ be a semiabelian $p$-group. Then there are abelian $p$-groups $A_1,A_2,..,A_r$ for which there is a rank preserving epimorphism $A_1\wr (A_2\wr ... \wr A_r)\ra G$. \end{thm}
By this we complete the solution of the minimal ramification problem for semiabelian $p$-groups.

I would like to thank my advisor Jack Sonn for his valuable advice and comments, for pointing out an error in the first draft of this paper and suggesting a way to correct it. I would also like to thank the referee for many helpful remarks that improved this paper.

\section{Properties of decompositions}

The family of semiabelian groups has appeared in many forms in problems that arise from field theory  %like %in the contexts of
(e.g. geometric realizations \cite{Den},\cite{M},\cite{Sal}, generic extensions \cite{Sal} and the minimal ramification problem \cite{KS}).
The following notion of a decomposition is used in \cite{Den} to characterize semiabelian groups:
\begin{defn} Let $G$ be a non-trivial group. A {\it decomposition} of $G$ is an abelian normal subgroup $A\lhd G$ and a proper subgroup $H<G$ such that $G=AH$.
\end{defn}
Dentzer (\cite{Den}) showed that a non-trivial group $G$ is semiabelian if and only if there is a decomposition $G=AH$
where $H$ is semiabelian.

In the following lemma we summarize several properties of such decompositions which we shall use repeatedly.  Let $\Phi(G)$ denote the Frattini subgroup of a group $G$. Recall that if $G$ is a $p$-group then $\Phi(G)=G^p[G,G]$.
\begin{lem}\label{basic.lem} Let $G$ be a $p$-group with decomposition $G=AH$.
Let $\oline{A}=A/A^p[A,H]$, $\pi:A\ra \oline{A}$ the natural map and $M$ a minimal subgroup of $A$ for which $\pi(M)=\oline{A}$. Then:
\begin{enumerate}
\item $[A,H]$ is a subgroup of $A$ that is normal in $G$,
\item
\label{structure_of_fratini} $\Phi(G)=A^p[A,H]\Phi(H),$ %\end{equation}
\item  $\oline{A}$ is a non-trivial elementary abelian $p$-group and $d(\oline{A})=d(M)$.
\end{enumerate}
If we assume in addition that \begin{equation}\label{intersect.equ}A\cap H\subseteq A^p[A,H]\cap \Phi(H),\end{equation} then:
\begin{enumerate}
 \item[(4)] \label{small cap} there is an isomorphism:
$$\psi:G/\Phi(G)\cong \oline{A}\times (H/\Phi(H)),$$
that is given explicitly for all $a\in A,h\in H$ by: $$\psi(ah\Phi(G))=(aA^p[A,H],h\Phi(H)).$$
\item[(5)] \label{equality-d} $d(G)=d(\oline{A})+d(H)$.
 \end{enumerate}
\end{lem}
\begin{proof}%[Proof of Lemma \ref{basic.lem}]
\begin{enumerate}
\item For $a\in A,h\in H$, we have $[a,h]=a^{-1}h^{-1}ah=a^{-1}a^h$ where $a^h:=h^{-1}ah\in A$ and hence $[A,H]\leq A$.

    Since $[A,H]$ is a subgroup of $A$ it is centralized by $A$. For $h,h'\in H$ and $a\in A$ we have $[a,h]^{h'}=[a^{h'},h^{h'}]\in [A,H]$ and hence $[A,H]^{h'}\subseteq [A,H]$. It follows that $[A,H]$ is normalized by $A$ and $H$ and hence is a normal subgroup of $G$.
\item Clearly $$A^p[A,H]\Phi(H)=A^p[A,H]H^p[H,H]\subseteq G^p[G,G]=\Phi(G).$$ To show the converse we prove that for $g_1=a_1h_1,g_2=a_2h_2\in G,$ where $a_1,a_2\in A$ and $h_1,h_2\in H$, the commutator $[g_1,g_2]=g_1^{-1}g_2^{-1}g_1g_2$ and $g_1^p$ are elements of $A^p[A,H]\Phi(H)$. We use the following identities:
$$ [x,yz]=[x,z][x,y]^z,$$
$$ [xy,z]=[x,z]^y[y,z],$$ where
$x^y=y^{-1}xy$. It follows that:
$$ [g_1,g_2]=[a_1h_1,g_2]=[a_1,g_2]^{h_1}[h_1,g_2].$$ Expanding the commutators on the right hand side we have: $$ [a_1,g_2]=[a_1,a_2h_2]=[a_1,h_2][a_1,a_2]^{h_2}=[a_1,h_2]\in [A,H],$$
$$ [h_1,g_2]=[h_1,a_2h_2]=[h_1,h_2][h_1,a_2]^{h_2}\in [H,H][A,H]^{h_2}.$$
Since by part (1), $[A,H]\lhd G$ we have:
$$ [g_1,g_2]=[a_1,h_2]^{h_1}[h_1,h_2][h_1,a_2]^{h_2}\in [A,H][H,H][A,H]=[A,H][H,H].$$
We therefore have $[G,G]\subseteq [A,H][H,H]$. To prove that $g_1^p\in A^p[A,H]\Phi(H)$ we use the equality:
$$ g_1^p=a_1^ph_1^p \text{ (mod [G,G])}.$$ It follows that  $g_1^p\in A^pH^p[G,G]\subseteq A^p[A,H]\Phi(H)$.
\item Since all $p$-powers of $A$ are in $A^p[A,H]$, $\oline{A}$ is an elementary abelian $p$-group. Assume on the contrary $\oline{A}=\{1\}$. Then $A$ is contained in $\Phi(G)$ and hence the equality $G=AH$ implies by \cite[Corollary 10.3.3]{H} that $G=H$. This contradicts the assumption that $H$ is a proper subgroup of $G$ as part of a decomposition.

Let us show that $d(\oline{A})=d(M)$.
Let $\oline{A}=\langle \oline{a}_1,\ldots,\oline{a}_d\rangle$ where $d:=d(\oline{A})$. Since $\pi(M)=\oline{A}$, $d(M)\geq d(\oline{A})$. Furthermore, $M$ contains elements $a_1,\ldots,a_d$ such that $\pi(a_i)=\oline{a}_i$ for $i=1,\ldots,d$. The subgroup $M':=\langle a_1,\ldots,a_d\rangle\leq M$ maps under $\pi$ onto $\oline{A}$ and hence by the minimality of $M$, $M=M'$. It follows that $d(M)=d$.

\item Let us first prove $G/\Phi(G)\cong \oline{A}\times H/\Phi(H)$ under the assumption $A\cap H=\{1\}$. In such case $G=A\rtimes H$. Since  by part (1) $A^p[A,H]\leq A$, we have $A^p[A,H]\cap \Phi(H)=\{1\}$ and hence part (2) shows $$\Phi(G)=A^p[A,H]\rtimes \Phi(H).$$ We therefore have:
    \begin{equation}\label{semidirect-dec.equ} G/\Phi(G)=\left(A\rtimes H\right)/\left(A^p[A,H]\rtimes \Phi(H)\right).\end{equation}

  As $A^p$ is characteristic in $A$ we have $A^p\lhd G$. Together with part (1) this implies that $A^p[A,H]$ is a normal subgroup of $G$. Therefore the right hand side of (\ref{semidirect-dec.equ}) is isomorphic to:
  \begin{equation*}\left(A\rtimes H/A^p[A,H]\rtimes\{1\}\right)/\left(A^p[A,H]\rtimes \Phi(H)/A^p[A,H]\rtimes\{1\}\right) \\ \end{equation*}
and hence to:
  \begin{equation}\label{semidirect-evloves.equ}
   ((A/A^p[A,H])\rtimes H)/(\{1\}\rtimes \Phi(H))=(\oline{A}\rtimes H)/(1\rtimes \Phi(H)).\end{equation}

Since $[A,H]=\{a^{-1}a^h|a\in A,h\in H\}$, $H$ acts trivially on $A/[A,H]$, and hence the actions in the semidirect products in (\ref{semidirect-evloves.equ}) are trivial. We therefore have an isomorphism: $$G/\Phi(G)\cong (\oline{A}\times H)/(\{1\}\times \Phi(H))\cong \oline{A}\times (H/\Phi(H)).$$

Let us prove the assertion without assuming further that $A\cap H=\{1\}$. Note that $A\cap H$ is a normal subgroup of $G$ and let $G_0=G/A\cap H, A_0=A/A\cap H, H_0=H/A\cap H$. Since $A_0\cap H_0=\{1\}$ we have $G_0=A_0\rtimes H_0$. In particular, the assertion holds for the decomposition $G_0=A_0H_0$. Thus,
\begin{equation}\label{semidirect-deduce.equ}G_0/\Phi(G_0)\cong \oline{A}_0\times H_0/\Phi(H_0),\end{equation} where $\oline{A}_0=A_0/A_0^p[A_0,H_0]$.
Since,
\begin{equation}\label{not-to-reg.equ} \begin{array}{l} \Phi(H_0)=H_0^p[H_0,H_0]=H^p[H,H]/A\cap H=\Phi(H)/A\cap H, \\
                    A_0^p[A_0,H_0]=A^p[A,H]/A\cap H,%\\
                    \end{array}\end{equation}
we have:
$$ \begin{array}{l}
\oline{A}_0=(A/A\cap H)/(A^p[A,H]/A\cap H)\cong A/A^p[A,H]=\oline{A},\\
H_0/\Phi(H_0)=(H/A\cap H)/(\Phi(H)/A\cap H)\cong H/\Phi(H).
\end{array}$$
By part (2) we also have:
$$ \Phi(G_0)=A_0^p[A_0,H_0]\Phi(H_0)=A^p[A,H]\Phi(H)/A\cap H=\Phi(G)/A\cap H. $$
We therefore have:
$$ \begin{array}{cl} G/\Phi(G) & \cong (G/A\cap H)/(\Phi(G)/A\cap H)= G_0/\Phi(G_0) \\
& \cong \oline{A}_0\times (H_0/\Phi(H_0))\cong \oline{A}\times H/\Phi(H). \end{array}$$

Since in all of the above isomorphisms the coset of $a\in A$ (resp. $h\in H$) passes to the coset of $a$ (resp. $h$) in the image, the resulting isomorphism $$\psi:G/\phi(G)\ra \oline{A}\times H/\Phi(H)$$ is given for all $a\in A,h\in H$ by: $$\psi(ah\Phi(G))=(aA^p[A,H],h\Phi(H)).$$

 \item
By part (4) we have $d(G/\Phi(G))=d(\oline{A}\times (H/\Phi(H)))$. Since $\oline{A}$ and $H/\Phi(H)$ are $p$-groups one has $$d(\oline{A}\times (H/\Phi(H)))=d(\oline{A})+d(H/\Phi(H)).$$  Recall that $d(H)=d(H/\Phi(H))$ (see e.g., the Basis Theorem in \cite[\S 5.4]{BH}). We therefore have $d(G)=d(\oline{A})+d(H)$.
\end{enumerate}
\end{proof}

By iterating Lemma \ref{basic.lem}.(4) we have:
\begin{cor}\label{iterative.cor}
Let $H_0$ be a $p$-group. Let $H_1\geq H_2\geq \ldots \geq H_k$ and $A_1,\ldots,A_k$ be subgroups of $H_0$ such that for each $i=1,\ldots,k$, $H_{i-1}=A_{i}H_{i}$ is a decomposition and $A_{i}\cap H_{i}\subseteq A_{i}^p[A_i,H_i]\cap \Phi(H_i)$. Let $\oline{A}_i:=A_i/A_i^p[A_i,H_i]$. Then there is an isomorphism
\begin{equation*}\label{decomposition in sequence} \psi:H_0/\Phi(H_0) \cong \left( \prod_{i=1}^k \oline{A}_i\right) \times H_k/\Phi(H_k) \end{equation*}
such that for all  $a_i\in A_i,i=1,...,k,$ and $h\in H_k$:
\begin{equation}\label{explicit.equ}\psi(a_1\ldots a_kh\Phi(H_0))=(a_1A_1^p[A_1,H_1],...,a_kA_k^p[A_k,H_k],h\Phi(H_k)).\end{equation}
\end{cor}
Note that since $H_0=A_1\cdots A_kH_k,$  every element $x\in H_0$ can be written as $a_1...a_kh$, for some $a_i\in A_i,i=1,...,k, h\in H_k$ and hence (\ref{explicit.equ}) provides a description of $\psi$ for all elements of $H_0/\Phi(H_0)$.

\section{Minimal decompositions}

A key ingredient in our proof of Theorem \ref{main.thm} is to find a decomposition $G=AH$ such that $d(G)=d(\oline{A})+d(H)$, where $\oline{A}=A/A^p[A,H]$.
 We shall prove this is the case for minimal decompositions.
\begin{defn}  Let $G$ be a semiabelian group.  A {\it minimal decomposition} of $G$ consists of the following data.

(1) a minimal normal abelian subgroup $A\lhd G$ for which there is a semiabelian proper subgroup $H'$ of $G$ satisfying $G=AH'$,

(2) a minimal semiabelian subgroup $H\leq G$  for which $G=AH$ (for the same $A$ given in (1)).
\end{defn}
By Dentzer's result (\cite{Den}) any non-trivial semiabelian group $G$ has a decomposition and therefore also a minimal one.

In order to apply Lemma \ref{basic.lem}.(5), we prove that minimal decompositions satisfy (\ref{intersect.equ}).
\begin{prop} \label{minimal decomposition}Let $G$ be a semiabelian $p$-group with a minimal decomposition $G=AH$. Then
$A\cap H\subseteq A^p[A,H]\cap \Phi(H).$
\end{prop}
\begin{proof}
We divide the proof into two parts: (1) $A\cap H\subseteq A^p[A,H]$, (2) $A\cap H\subseteq \Phi(H)$.

\begin{enumerate}\item  Let $\pi_A:A\ra \oline{A}$ where $\oline{A}:=A/A^p[A,H]$. Assume, on the contrary, that there is an $a\in A\cap H$ with non-trivial  image $\oline{a}:=\pi_A(a)$. By Lemma \ref{basic.lem}.(3), $\oline{A}$ is an elementary abelian $p$-group, and hence can be viewed as an $\F_p$-vector space.
We can choose an $\F_p$-subspace  $B_1 \subseteq \oline{A}$ such that $\oline{A}=\langle \oline{a}\rangle\oplus B_1$.

The group $A_1:=\pi_A^{-1}(B_1)$ is a proper subgroup of $A$. By definition of $\oline{A}$, $H$ acts trivially by conjugation on $\oline{A}$. Thus, as a preimage of an $H$-invariant group,  $A_1$ is $H$-invariant and hence a normal subgroup  of $G$.

We also have $\pi_A(A_1)\pi_A(A\cap H)=\oline{A}$ and therefore $$A_1(A\cap H)A^p[A,H]=A.$$ Since $A_1\supseteq A^p[A,H]$ this implies that $A_1(A\cap H)=A$ and therefore that $$A_1H=A_1(A\cap H)H=AH=G.$$ Thus, $A_1$ is a proper subgroup of $A$ that is normal in $G$ such that $G=A_1H$, contradicting the minimality of $A$.

\item Let us show that $A\cap H\subseteq \Phi(H)$ by induction on $|G|$. Assume that for any semiabelian group $G_0$ with $|G_0|<|G|$ and any minimal decomposition $G_0=BK$, we have $B\cap K\subseteq \Phi(K)$.

 Let $\pi_H:H\ra H/\Phi(H)$ and assume on the contrary there is an $a\in A\cap H$ for which $\wh{a}:=\pi_H(a)$ is non-trivial. Let $H_1=H$ and $H_{i}=A_{i+1}H_{i+1}$, $i=1,2,...,k-1,$ be a sequence of minimal decompositions such that $H_k$ is the first for which $\wh{a}\not\in \pi_H(H_k)$.

 Let $\oline{A}_i=A_i/A_i^p[A_i,H_i]$ and $\pi_i:A_i\ra \oline{A}_i$.
 By the induction hypothesis and part 1, $A_i\cap H_i\subseteq A_i^p[A_i,H_i]\cap \phi(H_i)$, for $i=2,\ldots,k.$ Thus, we can apply Corollary \ref{iterative.cor} and obtain an isomorphism
 \begin{equation*} \psi:H/\Phi(H) \cong \left( \prod_{i=2}^k \oline{A}_i\right) \times H_k/\Phi(H_k). \end{equation*} such that for all  $a_i\in A_i,i=2,...,k,h\in H_k$:
 \begin{equation}\label{explicit2.equ}\psi(\pi_H(a_2\ldots a_kh))=(\pi_2(a_2),...,\pi_k(a_k),\pi_H(h)).\end{equation}
 As  $H_{k-1}=A_kH_k$,  (3.1) implies:
\begin{equation}\label{image.equ}\begin{array}{l}
  \psi(\pi_H(H_{k-1}))=  \{1\}^{k-2}\times \oline{A}_k \times H_k/\Phi(H_k), \\
   \psi(\pi_H(H_{k}))=   \{1\}^{k-1} \times H_k/\Phi(H_k).
\end{array}\end{equation}
Write $a=a_2a_3...a_kh$ for $a_i\in A_i$ and $h\in H_k$, $i=2,...,k$. Since $\wh{a}\in \pi_H(H_{k-1})\setminus \pi_H(H_{k})$,
  (\ref{image.equ}) implies:
\begin{equation*} \psi(\wh{a})
\in \left(\{1\}^{k-2}\times \oline{A}_k\times H_k/\Phi(H_k)\right)\setminus \left(\{1\}^{k-1}\times H_k/\Phi(H_k)\right),  \end{equation*}
and hence by (\ref{explicit2.equ}), $\pi_k(a_k)\not=1$.

Let $\pi_k(a_k),x_1,..,x_r$ be a basis of $\oline{A}_k$ and let $A_k':=\pi_k^{-1}(\langle x_1,...,x_r\rangle)$. Then $A_k'$ is a proper subgroup of $A_k$ which is normal in $H_{k-1}$.
Since $\langle \pi_k(a_k)\rangle \pi_k(A_k')=\oline{A}_k$ and as $A_k'\supseteq A_k^p[A_k,H_k]$, we have $\langle a_k\rangle A_k'=A_k$.

The group $U_{k-1}:=A_k'H_k$ is a semiabelian subgroup of $H_{k-1}$. Since $A_k'$ is a proper subgroup of $A_k$ and $A_k$ is minimal we deduce that $U_{k-1}$ is a proper subgroup of $H_{k-1}$. Iteratively, define a semiabelian subgroup $U_i:=A_{i+1}U_{i+1}$ of $H_i$ for $i=1,..,k-2$. The decompositions $H_k=A_{i+1}H_{i+1}$ are minimal and hence each $U_i$ is a proper subgroup of $H_i$ for $i=1,...,k-1$.

We now claim that $AU_1=G$.  We have:
\begin{equation}\label{au1.equ} AU_1=AA_2...A_{k-1}A_k'H_k = AA_2...A_{k-1}H_kA_k'H_k.\end{equation}
Since \begin{equation*}
a_k=a_{k-1}^{-1}...a_2^{-1}ah^{-1}\in A_{k-1}...A_2AH_k=AA_2...A_{k-1}H_k,\end{equation*}
the right hand side of  (\ref{au1.equ}) contains: \begin{equation*} AA_2...A_{k-1}\langle a_k \rangle A_k' H_k=A_1...A_kH_k=G.
\end{equation*}
It follows that $G=AU_1$ which contradicts the minimality of $H_1=H$.
\end{enumerate}
\end{proof}

Combining Lemma \ref{basic.lem} and Proposition \ref{minimal decomposition} we have:
\begin{thm}\label{rank decomposition.thm} Let $G$ be a semiabelian $p$-group and $G=AH$ a minimal decomposition. Then $d(G)=d(\oline{A})+d(H)$. In particular $d(H)<d(G)$.
\end{thm}
\begin{proof} By Proposition \ref{minimal decomposition}, $A\cap H\subseteq A^p[A,H]\cap \Phi(H)$. By Lemma \ref{basic.lem}.(5) this implies $d(G)=d(\oline{A})+d(H)$.  By Lemma \ref{basic.lem}.(3), $\oline{A}$ is non-trivial and hence $d(H)<d(G)$.   \end{proof}

\section{Equality of families}

Let $G$ be a semiabelian $p$-group. We shall prove Theorem \ref{main.thm} in three steps. In Step I, we use a minimal decomposition $G=AH$ to construct a rank preserving epimorphism $\psi_1:A\rtimes H\ra G$. In Step II, we construct a subgroup $M\leq A$ and a rank preserving epimorphism $\psi_2:M\wr H \ra A\rtimes H$. In Step III, we prove Theorem~\ref{main.thm} by iterating steps I and II.

\subsection*{Step I} At first, we use  Theorem \ref{rank decomposition.thm} to prove the following Corollary:
\begin{cor}\label{rank preserving semidirect product} Let $G$ be a semiabelian $p$-group with a minimal decomposition $G=AH$. Let ${G_0}:=A\rtimes H$ with respect to the action induced by conjugation in $G$.
Then there is a rank-preserving epimorphism $\psi_1:{G_0} \ra G$. % defined by $\psi_1(a,h)=ah$ is rank preserving.
\end{cor}
\begin{proof} Let $\psi_1:{G_0} \ra G$ be defined by $\psi_1(a,h)=ah.$ It is an homomorphism since for all $a_i\in A,h_i\in H,i=1,2,$
$$ \psi_1(a_1,h_1)\psi_1(a_2,h_2) = a_1h_1a_2h_2=a_1a_2^{h_1^{-1}}h_1h_2=\psi_1((a_1,h_1)(a_2,h_2)).$$ Since $A$ and $H$ are in $\Im(\psi)$, $\psi$ is surjective. Let ${A_0}=A\rtimes 1,{H_0}=1\rtimes H\leq {G_0}$ and $\oline{A}_0=A_0/A_0^p[A_0,H_0]$. Since $${A_0}\cap {H_0}=\{1\}\subseteq {A_0}^p[{A_0},{H_0}]\cap \Phi({H_0}),$$ it follows from Lemma \ref{basic.lem}.(5) that  %\begin{equation*}
$d({G_0})=d(\oline{A}_0)+d({H_0}).$ %\end{equation*}
We have $H_0\cong H$ and since $$[{A_0},{H_0}]=\langle a^{-1}a^h | a\in A_0, h\in H_0\rangle\cong [A,H],$$ we also have $\oline{A}_0\cong \oline{A}$. In particular, $d(G_0)=d(\oline{A})+d(H)$. By Theorem \ref{rank decomposition.thm} $d(G)=d(\oline{A})+d(H)$ and hence $d({G_0})=d(G)$.
\end{proof}

\subsection*{Step II} We show that the group $G_0$ in Corollary \ref{rank preserving semidirect product} is a rank preserving epimorphic image of a corresponding wreath product. We first recall the following definition:
\begin{defn} Let $H$ be a group and $A$ an abelian group.
\begin{enumerate} \item Let $M_1^H(A)$ be the induced $H$-module, i.e. the abelian group of all functions $f:H\ra A$ with the $H$-action $f^h(x)=f(xh^{-1})$ for all $x,h\in H,f\in M_1^H(A)$.
\item For every $a\in A$, let $a_*\in M_1^H(A)$ be defined by $a_*(1)=a$ and $a_*(h)=1$ for $h\not=1$.
\item The wreath product $A\wr H$ is the semidirect product $M_1^H(A)\rtimes H$ with respect to the above $H$-action.
\end{enumerate}
\end{defn}
Note that for $a\in A,h\in H$, one has $a_*^h(x)=a$ if $x=h$ and $a_*^h(x)=1$ otherwise.

To compute the rank of wreath products we shall use the following well known lemma (see e.g. \cite{Y} or \cite{BK}).
\begin{lem}\label{degree.lem} Let $A$ and $H$ be $p$-groups and assume $A$ is abelian. Then $$d(A\wr H)=d(A)+d(H).$$
\end{lem}

\begin{prop}\label{wreath.prop} Let $G=A\rtimes H$, where $A$ and $H$ are $p$-groups and $A$ is abelian. Let $\oline{A}=A/A^p[A,H]$ and $\pi:A\ra\oline{A}$ the natural map. Let $M$ be a minimal subgroup of $A$ for which $\pi(M)=\oline{A}$. Then there is a rank preserving  epimorphism $$\psi_2:M\wr H\ra G.$$
\end{prop}
\begin{proof} Let $\psi_2$ be the epimorphism from $A\wr H$ to $G$ (see \cite[Lemma 16.4.3]{FJ}) defined for all $f\in M_1^H(A)$ and $h\in H$ by\footnote{In \cite{FJ}, $\psi_2$ was defined by $\psi_2(f,h)=\left(\prod_{\sigma \in H} f(\sigma)^{\sigma^{-1}}, h \right)$. The source of the difference is in the definition of $f^{\tau}$. In \cite{FJ}, $f^{\tau}(\sigma)=f(\tau\sigma)$.}:
$$ \psi_2(f,h)=\left(\prod_{\sigma \in H} f(\sigma)^{\sigma}, h \right).$$ We claim that the restriction of $\psi_2$ to $M\wr H$ remains surjective. As $\psi_2(1,h)=h$ for all $h\in H$, we have $H\subseteq \Im(\psi_2)$. Since for every $a\in M$, $\psi_2(a_*,1)=a$, we have $M\subseteq \Im(\psi_2)$.
As $\pi(M)=\oline{A}$ we have $MA^p[A,H]=A$ and hence $A^p[A,H]\langle M,H\rangle=G$. Since $A^p[A,H]$ is contained in $\Phi(G)$ this implies $\langle M,H\rangle=G$. It follows that $\Im(\psi_2)\supseteq G$ which proves the claim.
It remains to show that $d(M\wr H)=d(G)$. By Lemma \ref{degree.lem}, $d(M\wr H)=d(M)+d(H)$. By Lemma \ref{basic.lem}.(3), $d(M)=d(\oline{A})$. By Theorem \ref{rank decomposition.thm}, $d(G)=d(\oline{A})+d(H)$. It follows that:
$$ d(M\wr H)=d(M)+d(H)=d(\oline{A})+d(H)=d(G).$$
\end{proof}

\subsection*{Step III} Let $G$ be a semiabelian $p$-group. The composition $\psi_2\circ \psi_1$ gives a rank preserving epimorphism from $M\wr H\ra G$ where $M\leq G$ is abelian and $H< G$ is semiabelian. To prove Theorem \ref{main.thm} we iterate this process using the following well known lemma. For the sake of completeness we include a proof of this lemma.
\begin{lem}\label{functoriality with second argument}
Let $A$ be a finite abelian group and $\psi:G\ra \Gamma$ an epimorphism of finite groups.
Then there is an epimorphism $\pst:A\wr G\ra A\wr \Gamma$.
\end{lem}
\begin{proof}
We shall treat the $\Gamma$-module $M_1^\Gamma(A)$ as a $G$-module via the map $\psi$. By the Frobenius Reciprocity Theorem \begin{equation}\label{iso.equ} \Hom(A,M_1^\Gamma(A))\cong \Hom_G(M_1^G(A),M_1^\Gamma(A)),\end{equation} where $\Hom_G$ denotes the group of $G$-homomorphisms. The Isomorphism (\ref{iso.equ}) associates to a homomorphism $i:A\ra M_1^\Gamma(A)$, a homomorphism of $G$-modules \\$i^*:~M_1^G(A)\ra~M_1^\Gamma(A)$ that is given by:
\begin{equation*} i^*(f)=\prod_{g\in G} i(f(g))^{\psi(g)}. \end{equation*}
Let $i:A\ra M_1^\Gamma(A)$ be the homomorphism $i(a)=a_*. $ Then \begin{equation}\label{induced.equ} i^*(f)(\gamma)=\left(\prod_{g\in G}f(g)_*^{\psi(g)}\right)(\gamma)=\prod_{\{g\in G:\psi(g)=\gamma\}} f(g),\end{equation}
is the function that sums over the values of $f\in M_1^G(A)$ on the fiber $\psi^{-1}(\gamma)$.

We claim that $i^*$ is surjective. Let $f\in M_1^\Gamma(A)$. For every $\gamma\in \Gamma$, fix  an element $g_\gamma\in G$ for which $\psi(g_\gamma)=\gamma$ and define $\widetilde{f}=\prod_{\gamma\in \Gamma}f(\gamma)_*^{g_\gamma}\in M_1^G(A)$. In particular, $\widetilde{f}(g)=1$ for $g\not\in\{g_\gamma|\gamma\in \Gamma\}$ and $\widetilde{f}(g_\gamma)=f(\gamma)$.  % by $\widetilde{f}(g_\gamma)=f(\gamma)$ for all $\gamma\in \Gamma$. %and $\widetilde{f}(g)=1$ if $g\not=g_\gamma$ for all $\gamma\in \Gamma$.
By (\ref{induced.equ}), $i^*(\widetilde{f})(\gamma)=\widetilde{f}(g_\gamma)=f(\gamma)$ for all $\gamma\in \Gamma$. It follows that for every $f\in M_1^\Gamma(A)$, $i^*(\widetilde{f})=f$, proving the claim.

Since $i^*:M_1^G(A)\ra M_1^\Gamma(A)$ is an epimorphism of $G$-modules, $i^*$ induces an epimorphism $\psi_2':M_1^G(A)\rtimes G\ra M_1^\Gamma(A)\rtimes G$. Since $\ker(\psi)\lhd G$ acts trivially on $M_1^\Gamma(A)$, $\ker(\psi)$ is a normal subgroup of $M_1^\Gamma(A)\rtimes G$. In particular, $\psi_2'$ induces an epimorphism $$\psi_2:A\wr G=M_1^G(A)\rtimes G\ra (M_1^\Gamma(A)\rtimes G)/\ker(\psi)\cong M_1^\Gamma(A)\rtimes \Gamma= A\wr \Gamma.$$  \end{proof}

\begin{proof}[Proof of Theorem \ref{main.thm}] We argue by induction on $|G|$. Let $G=AH$ be a minimal decomposition. By Corollary \ref{rank preserving semidirect product}, there is a rank preserving epimorphism $\psi_1:A\rtimes H\ra G$. Let $\oline{A}= A/A^p[A,H]$, $\pi:A\ra \oline{A}$ the natural map and $A_1$ a minimal subgroup of $A$ for which $\pi(A_1)=\oline{A}$. By Proposition \ref{wreath.prop} there is a rank preserving epimorphism $\psi_2:A_1\wr H\ra A\rtimes H$. Thus $\psi=\psi_2\circ\psi_1:A_1\wr H\ra G$ is a rank preserving epimorphism.

By the induction hypothesis there are abelian $p$-groups $A_2,\ldots, A_r$ and a rank preserving epimorphism $\phi_0: A_2 \wr (A_3 \wr\ldots \wr A_r)\ra H$. By Lemma \ref{functoriality with second argument}, $\phi_0$ induces an epimorphism $\phi: A_1\wr (A_2\wr \ldots\wr A_r)\ra A_1\wr H$. Using the equality $d(A_2\wr \ldots \wr A_r)=d(H)$ and Lemma \ref{degree.lem} we have:
$$ d(A_1\wr H)=d(A_1)+d(H)=d(A_1)+d(A_2\wr\ldots \wr A_r)=d(A_1\wr\ldots \wr A_r)$$ and hence $\phi$ is rank preserving. It follows that $\psi\circ\phi:A_1\wr (A_2\wr \ldots \wr A_r)\ra G$ is a rank preserving epimorphism.
\end{proof}

As a corollary we have:
\begin{cor} The families $\mG_p$ and $\mS_p$ are equal.
\end{cor}
\begin{proof}
The inclusion $\mG_p\subseteq \mS_p$ follows from \cite{KS}. For abelian $p$-groups $A_1,\ldots,A_r$ the iterated wreath product $A_1\wr (A_2\wr \ldots A_r)$ is in $\mG_p$. By Theorem \ref{main.thm}, every semiabelian group is a rank preserving epimorphism of such an iterated wreath product and hence in $\mG_p$. Thus, $\mG_p=\mS_p$. \end{proof}


\begin{thebibliography}{9}
\bibitem{BH} {\sc Bechtell, H.,} \it Theory of groups, \rm Addison-Wesley Pub. Co., 1971. 
\bibitem{BK} {\sc Bu\'{z}asi, K., Kova\'{c}s, L. G.,} \it The minimal number of generators of wreath products of nilpotent groups, \rm Contemporary Mathematics, {\bf 93} (1989), 115--121.
\bibitem{Den} {\sc Dentzer, R.}, \it On geometric embedding problems and semiabelian groups. \rm { Manuscripta Math.} {\bf 86} (1995), no. 2, 199-216.
\bibitem{FJ} {\sc Fried, M.D., Jarden, M.,} \it Field arithmetic, \rm {vol. 11, 2nd edn. Revised and enlarged by Moshe Jarden Ergebnisse der Mathematik (3). Springer, Heidelberg, 2005.}

\bibitem{KS} {\sc Kisilevsky, H., Sonn, J.,} \it On the minimal ramification problem for $l$-groups. \rm { Comp. Math.}, {\bf 146} (2010), 599--606.

\bibitem{H}{\sc Hall, M.,}  The theory of groups, Second edition.

\bibitem{M} {\sc Matzat, B.H.,} \it Konstruktive Galoistheorie. \rm {Springer Verlag}, (1987).

\bibitem{P} {\sc Plans, B.,} \it On the minimal number of ramified primes in some $l$-extensions of $\mQ$, \rm { Pacific J. Math.} {\bf 215} (2004), 381-391.
\bibitem{Sal} {\sc Saltman, D.,} \it Generic Galois extensions and problems in field theory. \rm { Adv. in Math.} {\bf 43} (1982), no. 3, 250--283.
\bibitem{Y} {\sc Yeo Kok Chye,} \it Minimal number of generators of some classes of groups, \rm {Ph.D. Thesis}, Australian National University, 1972. [Abstract: Bull. Austrlian. Math. Soc. {\bf 9} (1973), 301--302.]
\end{thebibliography}
\end{document}